\documentclass[final]{siamltex}

\usepackage[usenames,dvipsnames]{xcolor}
\usepackage[disable]{todonotes}

\usepackage{amsmath} 
\usepackage{amssymb}  
\usepackage{graphics}
\usepackage{epsfig}
\usepackage{framed}
\usepackage{hyperref}
\usepackage{cite}

\newtheorem{theo}{Theorem}[section]  

\newtheorem{rema}[theo]{Remark}

\newtheorem{assumption}[theo]{Assumption}
\newtheorem{example}[theo]{Example}

\newcommand{\beq}{\begin{equation}}
\newcommand{\eeq}{\end{equation}}
\newcommand{\beqa}{\begin{eqnarray}}
\newcommand{\eeqa}{\end{eqnarray}}
\newcommand{\beqs}{\begin{equation*}}
\newcommand{\eeqs}{\end{equation*}}
\newcommand{\beqas}{\begin{eqnarray*}}
\newcommand{\eeqas}{\end{eqnarray*}}

\newcommand{\X}{\mathcal{X}}

\newcommand{\R}{\mathbb R}
\newcommand{\N}{\mathbb N}

\newcommand{\bigO}{{\cal O}}

\newcommand{\hesslb}{c}

\newcommand{\dist}{\mbox{dist}}

\newcommand{\Idn}{{I_n}}
\newcommand{\Cmnsub}{\mathcal{C}_{c,L}}
\newcommand{\Cmnsubinf}{\mathcal{C}_{c,L}}

\newcommand\Mtilde{\stackrel{\sim}{\smash{{M}}\rule{0pt}{1.1ex}}}
\newcommand{\hM}{\Mtilde}
\newcommand{\hf}{{\hat{f}}}
\definecolor{hlitecolor}{rgb}{0.15, 0.23, 0.89}

\title{Convergence Rate of Incremental Gradient and Incremental Newton Methods}

\author{
M.~G\"urb\"uzbalaban\thanks{Laboratory for Information and Decision Systems, Massachusetts Institute of Technology, Cambridge, MA, 02139, USA. email: \{mertg, asuman,parrilo\}@mit.edu.}
\and A.~Ozdaglar$^*$
\and P.~Parrilo$^*$
}
\begin{document} 
\maketitle
\begin{abstract} The incremental gradient method is a prominent algorithm for minimizing a finite sum of smooth convex functions, used in many contexts including large-scale data processing applications and distributed optimization over networks. It is a first-order method that processes the functions one at a time based on their gradient information. The incremental Newton method, on the other hand, is a second-order variant which exploits additionally the curvature information of the underlying functions and can therefore be faster. In this paper, we focus on the case when the objective function is strongly convex and present fast convergence results for the incremental gradient and incremental Newton methods under the constant and diminishing stepsizes. For a decaying stepsize rule $\alpha_k = \Theta(1/k^s)$ with $s \in (0,1]$, we show that the distance of the IG iterates to the optimal solution converges at rate $\bigO(1/k^{s})$ (which translates into $\bigO(1/k^{2s})$ rate in the suboptimality of the objective value). For $s>1/2$, this improves the previous $\bigO(1/\sqrt{k})$ results in distances obtained for the case when functions are non-smooth. We show that to achieve the fastest $\bigO(1/k)$ rate, incremental gradient needs a stepsize that requires tuning to the strong convexity parameter whereas the incremental Newton method does not. The results are based on viewing the incremental gradient method as a gradient descent method with gradient errors, devising efficient upper bounds for the gradient error to derive inequalities that relate distances of the consecutive iterates to the optimal solution and finally applying Chung's lemmas from the stochastic approximation literature to these inequalities to determine their asymptotic behavior. In addition, we construct examples to show tightness of our rate results. 
\end{abstract}

\section{Introduction}

We consider the following additive cost optimization problem
\beq\label{pbm-multi-agent}
\min \sum_{i=1}^m f_i(x) \quad \mbox{s.t.} \quad x \in \R^n
\eeq
where the objective function is the sum of a large number of convex {\it component functions} $f_i: \R^n\to \mathbb{R}$. Such problems arise in a number of settings including distributed optimization across 
$m$ agents, where the component function $f_i$ corresponds to the local objective function of agent $i$ \cite{Boyd2011AdmmBook,NedicOzdaglar2009,Nedic2007rate,RamNedicVeer2007}, and statistical estimation problems where each $f_i$ represents the loss function associated with one of the data blocks \cite{bertsekas2011incremental,BottouLecun2005,Dickenstein2014,Jordan13DistLearningApi}. Our goal is to exploit the additive structure of problem (\ref{pbm-multi-agent}) and solve it using incremental methods which involve sequential processing of component functions.

We first consider the {\it incremental gradient (IG) method} for solving problem (\ref{pbm-multi-agent}). IG method is similar to the standard gradient method with the key difference that at each iteration, the decision vector is updated incrementally by taking sequential steps  along the gradient of  the component functions $f_i$ in a cyclic order. Hence, we can view each outer iteration $k$ as a cycle of $m$ \emph{inner iterations}: starting from initial point $x_1^1 \in \R^n$, for each $k\ge1$, we update the iterate $x_i^k$ as
 \begin{equation}\label{eq-inner-update} x_{i+1}^k :={x_i^k - \alpha_k \nabla f_i (x_i^k)} , \qquad i=1,2,\dots,m,
   \end{equation}
where $\alpha_k>0$ is a stepsize. We set $x_1^{k+1} = x_{m+1}^k$ and refer to $\{x_1^k\}$ as the {\it outer iterates}. 
  When the component functions are not smooth, we can replace gradients with a subgradient and the corresponding method is called the incremental subgradient method. 
  Using the update relation (\ref{eq-inner-update}), for each $k\ge 1$, we can write down the relation between the outer iterates as
 \begin{equation}\label{eq-outer-update}
  x_1^{k+1} = x_1^k - \alpha_k \sum_{i=1}^m \nabla f_i (x_i^k),
  \end{equation}  
 where  $\sum_{i=1}^m \nabla f_i (x_i^k)$ is the aggregated component gradients and serve as an approximation to the full gradient $\nabla f(x_1^k)$ with the difference that it is evaluated at different inner iterates. 
  
IG is a prominent algorithm with a long history that has appeared in many contexts. In the artificial intelligence literature, it has been used in training neural networks in the 80s and is known as the online backpropagation algorithm \cite{widrowh60,Luo1991convergence,bertsekas2011incremental}. Another well-known example of this method is the Kaczmarz method for solving systems of linear equations, which is a special case of the IG method \cite{Bertsekas15Book}.

Due to the simplicity and long history of IG method, its global convergence has been supported under various conditions (see \cite{bertsekas2011incremental} for a survey), however characterizing its convergence rate has been the subject of more recent work . Among the papers relevant to our work, Kohonen \cite{Kohonen74} focused on quadratic component functions with constant stepsize, $\alpha_k=\alpha>0$ for all $k$, and showed that the iterates may converge to a limit cycle (subsequence of inner iterates converge to different limits close to optimal). The papers \cite{Luo1991convergence,Gaivoronski94, Grippo94,ManSol98IncGradMomentum, TsengIncrGradient98, BertTsit2000Gradient, TsengLuo2008, Bertsekas1996incremental,bertsekas1997hybrid} focused on diminishing stepsize and showed convergence of the algorithm and its variants under different assumptions. The papers \cite{Solodov98IncrGrad} and \cite{NedicBert2001IncSubgrad} studied IG with a constant stepsize and under different assumptions on the component functions, and showed that the iterates converge to a neighborhood of the optimal solution (where the size of the neighborhood is a function of the stepsize). Most closely related to our paper is a convergence rate result provided by N\'edic and Bertsekas\cite{NedicBert2001IncSubgrad}, which under a strong-convexity type condition on the sum function $f(x) = \sum_{i=1}^m f_i(x)$, but without assuming differentiability of the component functions, shows that the distance of the iterates generated by the incremental subgradient method converges at rate $\bigO({1\over \sqrt{k}})$ to the optimal solution with a  properly selected diminishing stepsize. 

Luo \cite{Luo1991convergence} considered a special case of the problem \eqref{pbm-multi-agent} in dimension one when there are two convex quadratic component functions with an identical non-zero curvature and showed that IG iterates converge in this particular case at rate $\bigO({1\over k})$ to the optimal solution. Motivated by this example, in this paper we show that N\'edic and Bertsekas' $\bigO({1\over \sqrt{k}})$ result can be improved when the component functions are smooth. In particular, when the component functions are quadratics and the sum function $f(x)$ is strongly convex, we first prove that the distances of the  iterates generated by the IG method converge at rate $\bigO({1\over k})$  (which translates into $\bigO({1\over k^2})$ in function values by the strong convexity of $f$). Then, we generalize this result to twice continuously differentiable component functions under some assumptions. Achieving this rate with IG requires using a diminishing stepsize that adapts to the strong convexity constant $c$ of the sum function, i.e., a stepsize that takes the form $R/k$ where $R>1/c$.\footnote{We note that a consequence of a paper by Hazan and Kale \cite{Hazan:2007bq} is that when each of the component functions is strongly convex, IG with iterate averaging and stepsize $\alpha_k = R/k$ where $R$ is the multiplicative inverse of the strong convexity constant, converges at rate $\bigO(\log k/k)$ in the suboptimality of the function value. However, the rate we obtain in this paper with a similar stepsize corresponds to $\bigO(1/k^{2})$ in the suboptimality of the objective value which is much faster.} We then consider alternative ``robust" stepsizes $\alpha_k = \Theta({1\over k^s})$ for $s\in (0,1)$, which does not require knowledge of the strong convexity constant,  and show that IG method with these stepsizes achieve a rate $\bigO({1\over k^s})$ in distances (which translates into $\bigO({1\over k^{2s}})$ in function values). We also provide lower bounds showing that these rates cannot be improved using IG. Furthermore, our results play a key role in the recently obtained convergence results for the \textit{random reshuffling} (RR) method \cite{GurRand15RR}. The random reshuffling method is a stochastic variant of IG where the order to visit the functions is selected as a random permutation of $\{1,2,\dots,m\}$ at the beginning of each cycle instead of the deterministic fixed order $\{1,2,\dots,m\}$ of IG (hence the name RR refers to the random reshuffling of the order). The convergence rate of the random reshuffling method has been a long standing open question. Fundamental to the analysis in \cite{GurRand15RR} which solves this open problem under some technical assumptions on the objective function $f$ is the fast convergence results introduced in this paper.

While IG achieves rate $\bigO({1\over k})$ in distances with stepsize ${R\over k}$, our rate estimates show dependence on strong convexity and Lipschitz constant of the sum function. We next consider an {\it incremental Newton (IN) method} introduced in \cite{GurOzPa15} for solving problem (\ref{pbm-multi-agent}) which scales the gradients of the component functions with an estimate of the Hessian of the sum function $f(x)$: starting from initial point $x_1^1\in\R^n$, initial stepsize $\alpha_1 > 0$ and initial Hessian estimate $H_0^1 = \Idn$, for each $k\geq 1$, IN method updates the iterate $x_i^k$ as 
	\begin{equation} x_{i+1}^k := x_i^k - \alpha_k (\bar{H}_i^k)^{-1} \nabla f_i (x_i^k) 
		\label{eq-inner-update-2}     
    \end{equation}
 where
       \begin{equation} H_i^k := H_{i-1}^k + \nabla^2 f_i(x_i^k), \quad  \bar{H}_i^k = H_i^k/k, 
       		\label{eq-hessian-update}
       \end{equation}
with the convention that $x_1^{k+1} = x_{m+1}^k$ and $H_0^{k+1}=H_m^k$. For IN, we provide rate estimates which do not depend on the Lipschitz constant and show that the IN method, unlike IG, converges with rate $\bigO({1\over k})$ without the necessity to adjust to the strong convexity constant.


\paragraph{\textbf{Notation}} For non-negative sequences $\{a_k\}$ and $\{b_k\}$, we write $a_k \geq \Omega(b_k)$ if there exists a real constant $h>0$ and a finite integer $k_0$ such that $a_k \geq h b_k$ for every $k\geq k_0$. The norm $\| \cdot\|$ denotes Euclidean norm for vectors and the spectral norm for matrices. We also write $a_k \geq \tilde{\Omega}(b_k)$ if there exists a real constant $h>0$ and infinitely many $k$ such that $a_k \geq h b_k$ is true\footnote{The $\tilde{\Omega}$ function defined here was introduced by Littlewood and Hardy in 1914. It is a weaker alternative to the $\Omega$ function and satisfies $a_k \geq \Omega(b_k) \implies a_k \geq \tilde{\Omega}(b_k)$ but not vice versa.}. The matrix $\Idn$ denotes the $n\times n$ identity. The sets $\R_+$ and $\N_+$ denotes the positive real numbers and positive integers respectively. We refer to twice continuously differentiable functions on $\R^n$ as \textit{smooth} functions. 

\section{Preliminaries}
We introduce the following lemma, known as Chung's lemma, which we will make use of in our rate analysis. The proof of the part $(i)$ of this lemma can be found in \cite[Sec 2.2]{Pol87}. For the proof of the part $(ii)$ of this lemma, we refer the reader to \cite[Lemma 4]{Chung:1954iy}.
\begin{lemma} \label{lem-chung-1} Let $u_k\geq 0$ be a sequence of real numbers. Assume there exists $k_0$ such that 
$$ u_{k+1} \leq (1-\frac{a}{k^s}) u_k + \frac{d}{k^{s+t}}, \quad \forall k\geq k_0,$$
where $0<s\leq 1$, $d>0$, $a>0$ and $t>0$ are given real numbers. Then, 
\begin{itemize}
	\item [$(i)$] If $s=1$, then
	      \begin{alignat*}{3}
   &\limsup_{k\to\infty} k^t u_k &~\leq~& \frac{d}{a-t} \quad  \quad  &\mbox{for}&\quad a>t,\\
	&\limsup_{k\to\infty} \frac{k^{a}}{\log k} u_k &~<~& \infty \quad &\mbox{for}& \quad a=t,\\
	&\limsup_{k\to\infty} k^a u_k &~<~& \infty \quad &\mbox{for}&\quad a<t.
\end{alignat*}
    \item [$(ii)$] If $0<s<1$, then\footnote{Part $(ii)$ of Lemma \ref{lem-chung-1} is still correct when $u_k$ is allowed to take negative values. However, this will not be needed in our analysis.} 
         $$ \limsup_{k \to \infty} k^{t} u_k \leq \frac{d}{a}.$$
\end{itemize}
\end{lemma} 

\section{Convergence rate analysis for IG}
\subsection{Rate for quadratic functions} We first analyze the convergence behavior when the component functions are quadratic functions before proceeding to the more general case when functions are twice continuously differentiable. 
Let $f_i(x): \R^n \to \R$ be a quadratic function of the form
    \beq f_i(x) = \frac{1}{2} x^T P_i x - q_i^T x + r_i, \quad i=1,2,\dots, m,
    	\label{def-fi}
    \eeq
where $P_i $ is a symmetic $n\times n$ square matrix, $q_i \in \R^n$ is a column vector and $r_i$ is a real scalar. The gradient and the Hessian of $f$ are given by
   \beq 
   		\nabla f_i(x) = P_i x - q_i, \quad \nabla^2 f_i(x) = P_i. 
   		\label{def-quad-grad-hess}
   	\eeq
The sum $f$ is also a quadratic which we next assume to be strongly convex.

\begin{assumption}\label{assump-sum-is-str-cvx} The sum function $f(x)$ is strongly convex on $\R^n$, i.e. there exists a constant $c>0$ such that the function $f(x) - \frac{c}{2} \| x\| ^2$ is convex on $\R^n$. 
\end{assumption}

	Under this assumption, the optimal solution to the problem \eqref{pbm-multi-agent} is unique, which we denote by $x^*$. In the particular case, when each $f_i$ is a quadratic function given by \eqref{def-quad-grad-hess}, then the Hessian matrix of the sum satisfies
  \beq
	   		 P: = \nabla^2 f(x) = \sum_{i=1}^m P_i \succeq c \Idn \succ 0,
	   		 \label{def-P}
	    \eeq	
and the optimal solution is
  \beq 
  		x^* = P^{-1}\sum_{i=1}^m q_i. 
   	\label{explicit-min-for-quads}   
   \eeq
The inner iterations of IG become
   $$ x_{i+1}^k = (\Idn-\alpha_k P_i) x_i^k + \alpha_k q_i, \quad i=1,2,\dots,m.$$
Therefore, the outer iterations are given by
  \beqa x_1^{k+1} &=& \prod_{i=1}^m (\Idn - \alpha_k P_i) x_1^k + \alpha_k \sum_{i=1}^m \prod_{j=i+1}^m (\Idn - \alpha_k P_j) q_i \label{outer-iter-expr-1} \\
                            &=& \bigg(\Idn - \alpha_k P  + \bigO(\alpha_k^3)\bigg)x_1^k + \alpha_k \sum_{i=1}^m q_i + \alpha_k^2 T(\alpha_k) + \bigO(\alpha_k^3) \label{outer-iter-quadratics-1} \\
             				  &=&  \big(\Idn - \alpha_k P \big) x_1^k + \alpha_k \sum_{i=1}^m q_i + \alpha_k^2 E(\alpha_k)  \label{outer-iter-quadratics-2}          				  
  \eeqa  
where 
		\beqa
				T(\alpha_k) &=& \sum_{1\leq i < j\leq m}P_j (P_i x_1^k -  q_i) =  \sum_{1\leq i < j\leq m} P_j \nabla f_i(x_1^k) \label{def-T-alphak} \\
				 E(\alpha_k) &=& T(\alpha_k) + \bigO(\alpha_k) + \bigO(\alpha_k x_1^k). \label{def-R-alphak}
		\eeqa  	
From \eqref{explicit-min-for-quads}, \eqref{outer-iter-quadratics-1} and \eqref{outer-iter-quadratics-2}, it follows that
	 \beqa x_1^{k+1} - x^* &=&  \big(\Idn - \alpha_k P + \bigO(\alpha_k^3) \big) (x_1^k - x^*) + \alpha_k^2  T(\alpha_k) + \bigO(\alpha_k^3). \label{dist-recursion-T}\\
      x_1^{k+1} - x^* &=&  \big(\Idn - \alpha_k P \big) (x_1^k - x^*) + \alpha_k^2  E(\alpha_k) \nonumber. 
     \eeqa
Taking norms of both sides of the last expression, defining 
		$$ \dist_k = \|x_1^k - x^* \| $$ 
as the distance to the optimal solution and using the lower bound \eqref{def-P} on the eigenvalues of $P$, we obtain
     \beqa  \dist_{k+1} &\leq&  \big\| \Idn - \alpha_k P \big\| \dist_k + \alpha_k^2 \| E(\alpha_k)\| \nonumber \\
     								 &\leq&   (1 - \alpha_k c)\dist_k + \alpha_k^2 \| E(\alpha_k)\| 
     								 \quad \mbox{if} \quad \alpha_k \| P \| \leq 1. \label{dist-recursion-quadratics} \\
     							&\leq&  (1 - \alpha_k c)\dist_k + \alpha_k^2 M_\infty
     							 \quad \mbox{if} \quad \alpha_k \| P \| \leq 1 \label{def-Minfty}
     \eeqa
where 
	$$M_\infty :=  \sup_{k\geq 1} \|E(\alpha_k)\|. $$ 
The next theorem is based on analyzing this recursion and is on the achievable convergence rate when the component functions are quadratics.

\begin{theorem}\label{theo-rate-quadratics} Let each $f_i(x) = \frac{1}{2} x_i^T P_i x - q_i^T x + r_i$ be a quadratic function as in \eqref{def-fi} for $i=1,2,\dots,m$. Suppose Assumption \ref{assump-sum-is-str-cvx} holds. Consider the iterates $\{x_1^k\}$ generated by the IG method with stepsize $\alpha_k = R/k^s$ where $R>0$ and  $s \in [0,1]$. Then, we have the following:
	\begin{itemize}
	\item [$(i)$] If $s=1$, then
	      \begin{alignat*}{3}
   &\limsup_{k\to\infty} k \dist_k &~\leq~& \frac{R^2 M}{Rc-1} \quad  \quad  &\mbox{for}&\quad R> 1/c,\\
	&\limsup_{k\to\infty} \frac{k}{\log k} \dist_k &~<~& \infty \quad &\mbox{for}& \quad R=1/c,\\
	&\limsup_{k\to\infty} k^{Rc} \dist_k &~<~& \infty \quad &\mbox{for}&\quad R<1/c.
\end{alignat*}
where $M=\lim_{k\to\infty} \| E(\alpha_k)\| = \big\| \underset{1\leq i < j\leq m}{\sum} P_j \nabla f_i(x^*)\big\|$.
    \item [$(ii)$] If $0<s<1$, then 
         $$ \limsup_{k \to \infty} k^{s} \dist_k \leq \frac{R M}{c}.$$
    \item [$(iii)$] If $s=0$ and $R\leq \frac{1}{\|P\|}$, then 
    			  \beqa \dist_{k+1} \leq (1 - c\alpha)^k dist_1 +  \frac{\alpha M_\infty}{c}, \quad \forall k\geq 1,	\label{const-step-recursion}
   \eeqa
where the stepsize $\alpha=\alpha_k = R$ is a constant and $M_\infty$ is defined by \eqref{def-Minfty}. 
\end{itemize}

\end{theorem}
\begin{proof} 
We first prove parts $(i)$ and $(ii)$. So, assume $0 < s \leq 1$. Plugging the expression for the stepsize into \eqref{dist-recursion-T} and taking norms, we obtain 
	\beq \dist_{k+1} \leq  \big(1 - \frac{Rc}{k^s} + \bigO(\frac{1}{k^{3s}})\big)\dist_k + \frac{R^2}{k^{2s}} \| T(\alpha^k)\|  + \bigO(\frac{1}{k^{3s}}).
	\label{dist-recursion-quad-2}
	\eeq
It is also easy to see from \eqref{def-T-alphak} that   
   \beqa 
   		\| T(\alpha_k)\| &\leq& \big\| \sum_{1\leq i < j\leq m} P_j \nabla f_i(x^*) \big\|+ \big\|\sum_{1\leq i < j\leq m} P_j \big(\nabla f_i(x_1^k) - \nabla f_i(x^*)\big) \big\|   \nonumber \\
   								  &=& \big\| \sum_{1\leq i < j\leq m} P_j \nabla f_i(x^*) \big\|+ \big\|\sum_{1\leq i < j\leq m} P_j P_i (x_1^k - x^*) \big\|  \nonumber \\
   								  &\leq& M + h_1 \dist_k  \label{bound-on-Ealphak}
   	\eeqa
for some finite positive constant $h_1$ (that depends on $\{P_i\}_{i=1}^m$ and $\{\nabla f_i(x^*)\}_{i=1}^m$). Then, from \eqref{dist-recursion-quad-2}, 
   $$ \dist_{k+1} \leq  \big(1 - \frac{Rc}{k^s} + \frac{R^2 h_1}{k^{2s}} + \bigO(\frac{1}{k^{3s}})\big)\dist_k + \frac{R^2 M}{k^{2s}} + \bigO(\frac{1}{k^{3s}}). 
   $$
Finally, applying Lemma \ref{lem-chung-1} with a choice of $0<a<Rc$, $t=s$ and $d>R^2M$ and letting $a\to Rc$ and $d\to R^2 M$ proves directly parts $(i)$ and $(ii)$.
To prove part $(iii)$, assume $s=0$ and $R\leq \frac{1}{\|P\|}$. Then, the stepsize $\alpha_k = \alpha = R$ is a constant and by \eqref{dist-recursion-quadratics}, for all $k\geq 1$, 
    $$\dist_{k+1} \leq (1 - \alpha c)\dist_k + \alpha^2 M_\infty. $$ 
From this relation, by induction we obtain for all $k\geq 1$,
\beqas \dist_{k+1} \leq (1 - c\alpha)^k \dist_1 +  \alpha^2 M_\infty \sum_{j=0}^{k-1} (1-c\alpha )^j.
   \eeqas
 As the geometric sum $\sum_{j=0}^{k-1} (1-c\alpha )^j \leq \frac{1}{c\alpha}$ for all $k\geq 1$, this proves part $(iii)$.    
\end{proof}

\begin{rema} Under the setting of Theorem \ref{theo-rate-quadratics}, in the special case when $x^*$ is a global minimizer for each of the component functions, we have $\nabla f_i(x^*) = 0$ for each $i=1,2,\dots,m$. This implies that $M=0$ and for stepsize $R/k$ with $R>1/c$ we have $\underset{k\to\infty}{\limsup} ~k \dist_k = 0$, i.e. $\dist_k = o(1/k)$. This rate can actually be improved further as follows. In this special case, assume for simplicity that $x^* = 0$ (the more general case can be treated similarly by shifting the coordinates and considering the functions $ f_i(x-x^*)$ and $f(x-x^*)$). Then, this implies that $q_i = 0$ for all $i$ and therefore from \eqref{outer-iter-expr-1} we have,
      \beqas \dist_{j+1} &=&  \bigg \|\prod_{i=1}^m \big(\Idn - \frac{R}{k} P_i \big)\bigg\| \dist_j = \bigg \|\Idn - \frac{R}{j} P + \bigO(1/j^2)\bigg\| \dist_j \\
                                  &\leq& \big(1-\frac{Rc}{j} + \bigO(1/j^2)\big) \dist_j \leq (1-\frac{\delta}{j})\dist_j, 
      \eeqas
where the last inequality holds for any $1 < \delta <Rc $ and $j$ large enough. As $\prod_{j=2}^k \big(1-\frac{\delta}{j}\big) \approx  \prod_{j=2}^k \big(1-\frac{1}{j}\big)^\delta = 1/k^\delta$, it follows that $ \dist_k = \bigO(1/k^\delta)$ for any $1 < \delta < Rc$ which is a stronger statement than $\dist_k = o(1/k)$.   
 \end{rema}
\subsection{Rate for smooth component functions}
In addition to Assumption \ref{assump-sum-is-str-cvx} on the strong convexity of $f$, we adopt the following assumptions that have appeared in a number of papers in the literature for analyzing incremental methods including \cite{AlgEkfs2003}, \cite{Bertsekas1996incremental}, \cite{GurOzPa15}. 


\begin{assumption}\label{assum-C2} The functions $f_i$ are convex and twice continuously differentiable on $\R^n$ for each $i=1,2,\dots,m$.\todo{\hfill Relax to \\ \hfill Lipschitz \\ \hfill gradients?}

\end{assumption}
\todo{\hfill Relax \\ \hfill convexity \\ \hfill assumption?}
\begin{assumption}\label{assum-iters-bdd} The iterates $\{x_1^k, x_2^k, \dots, x_m^k\}_{k\geq 1}$ are uniformly bounded, i.e. there exists a non-empty compact  Euclidean ball $\mathcal{X} \subset \R^n$ that contains all the iterates.
\end{assumption}

A consequence of these two assumptions is that by continuity of the first and second derivatives on the compact set $\mathcal{X}$, the first and second derivatives are bounded. Therefore there exists a constant $G$ such that 
  \beq 
  		\max_{1\leq i \leq m}\sup_{x \in \mathcal{X}} \| \nabla f_i(x)\| \leq G
   \label{def-G}   
   \eeq 
and there exists constants $L_i:=\max_{z\in \mathcal{X}} \|\nabla^2 f_i(z)\| \geq 0$ such that 
        \beqa \| \nabla f_i(x) - \nabla f_i(y) \|  &\leq& L_i \|x-y\|, \quad \mbox{for all} \quad x\in\mathcal{X}, \quad i=1,2,\dots,m. \label{eq-lip-constant}
   \eeqa
From the triangle inequality, $f$ has also Lipschitz gradients on $\X$ with constant 
\beqa        
         L=\sum_{i=1}^m L_i \label{def-L}.
\eeqa 
Another consequence is that, an optimal solution to the problem \eqref{pbm-multi-agent}, which we denote by $x^*$, exists and is unique by the strong convexity of $f$.  Furthermore, these two assumptions are sufficient 
for global convergence of both the incremental Newton and the incremental gradient methods to $x^*$ (see \cite{GurOzPa15}, \cite{Bertsekas99nonlinear} for a more general convergence theory). 
%
%
\subsubsection{Analyzing IG as a gradient descent with errors}\label{subsub-ig-as-grad-descent}
We can rewrite the inner iterations \eqref{eq-inner-update} more compactly as \begin{equation} x_1^{k+1} = x_1^k - \alpha_k \bigl( \nabla f (x_1^k) - e^k \bigr), \quad k\geq 1, \quad i=1,2,\dots,m,
		\label{outer-iter-perturbed-gradients}
   \end{equation}
where the term
   \beq\label{grad-error-incr-grad} e^k = \sum_{i=1}^m \bigl( \nabla f_i(x_1^k) -  \nabla f_i(x_i^k) \bigr)
   \eeq
can be viewed as the gradient error. If Assumption \ref{assum-C2} holds, we can substitute
    $$ \nabla f(x_1^k) = A_k (x_1^k - x^*) $$
into \eqref{outer-iter-perturbed-gradients} where $A_k =\displaystyle \int_0^1 \nabla^2 f(x^*+\tau (x^k-x^*)) d\tau$ is an average of the Hessian matrices on the line segment $[x_1^k, x^*]$ to obtain
 \begin{equation} x_1^{k+1} - x^* = (\Idn - \alpha_k A_k) (x_1^k - x^*) + \alpha_k e^k , \quad k\geq 1, \quad i=1,2,\dots,m.
   \end{equation}
Taking norms of both sides, this implies that
\beq \dist_{k+1} \leq \| \Idn - \alpha_k A_k \| \dist_k + \alpha_k \|e^k \|. \label{incr-grad-dist-difference-eqn}
\eeq
These relations show that the evolution of the distance to the optimal solution is controlled by the decay of the stepsize $\alpha_k$ and the gradient error $\|e^k\|$.  
This motivates deriving tight upper bounds for the gradient error. Note also that under Assumptions \ref{assump-sum-is-str-cvx}, \ref{assum-C2} and \ref{assum-iters-bdd},  the Hessian of $f$ and the averaged Hessian matrix $A_k$ admit the bounds 
   \beq c\Idn \preceq \nabla^2 f(x), A_k \preceq L\Idn, \quad x\in \X.
   \label{hessian-bounds}
   \eeq 
(see also \eqref{def-L}).   
The gradient error consists of the difference of gradients evaluated at different inner steps (see \eqref{grad-error-incr-grad}). This error can be controlled by the Lipschitzness of the gradients as follows: For any $k\geq 1$, 
%
\beqa \| e^k \|  &\leq& \sum_{i=2}^m L_i \| x_1^k - x_i^k  \| \leq \sum_{i=2}^m L_i  \sum_{j=1}^{i-1} \|  x_j^k -  x_{j+1}^k  \|  \nonumber \\
                        &\leq& \sum_{i=2}^m L_i  \alpha_k \sum_{j=1}^{i-1} \| \nabla f_j(x_j^k)\| \nonumber \\
                        &\leq& \alpha_k \hM,    \label{incr-grad-rate-recursion}
\eeqa
where 
		\beq 
		\hM := L  G m, 
		\label{def-M}
	\eeq
$L$ is a Lipschitz constant for the gradient of $f$ as in \eqref{def-L} and $G$ is an upper bound on the gradients as in \eqref{def-G}. Finally, plugging this into \eqref{incr-grad-dist-difference-eqn} and using the bounds \eqref{hessian-bounds} on the eigenvalues of $A_k$,
\beqa \dist_{k+1} &\leq& \max(\|1-\alpha_k c\|, \|1- \alpha_k L\|)\dist_k + \alpha_k^2 \hM \nonumber \\ 
                          &\leq& (1-\alpha_k c) \dist_k + \alpha_k^2 \hM \quad \mbox{if} \quad \alpha_k L \leq 1. \label{max-eigval-recursion}
\eeqa 
This is the analogue of the recursion \eqref{def-Minfty} obtained for quadratics with the only difference that the constants $M_\infty$ and $\|P\|$ are replaced by their analogues $\hM$ and $L$ respectively. Then, a reasoning along the lines of the proof of Theorem \ref{theo-rate-quadratics} yields the following convergence result which generalizes Theorem \ref{theo-rate-quadratics} from quadratic functions to smooth functions just by modifying the constants properly. We skip the proof for the sake of brevity.
\begin{theorem}\label{theo-ig-rate-general}  Let $f_i(x):\R^n \to \R$, $i=1,2,\dots,m$ be  component functions satisfying Assumptions \ref{assump-sum-is-str-cvx} and \ref{assum-C2}. Consider the iterates $\{x_1^k, x_2^k, \dots, x_m^k\}_{k\geq 1}$ obtained by the IG iterations \eqref{eq-inner-update} with a decaying stepsize $\alpha_k = R/k^s$ where $R>0$ and  $s \in [0,1]$. Suppose that Assumption \ref{assum-iters-bdd} is also satisfied. Then, all the conclusions in parts $(i)$, $(ii)$ and $(iii)$ of Theorem \ref{theo-rate-quadratics} remain valid by replacing $M$ and $M_\infty$ with $\hM$ and replacing $\|P\|$ with $L$.
\end{theorem}
\begin{rema}\label{rema-compare-constants-to-quad-case} Under the conditions of Theorem \ref{theo-rate-quadratics}, the quadratic functions $f_i$ have Lipschitz continuous gradients with constants $L_i = \|P_i\|$. Therefore,
	$$ 
		 M \leq \underset{1\leq i \leq m}{\sum} \sum_{j=i+1}^m L_j ~\| \nabla f_i(x^*)\|  \leq  \underset{1\leq i \leq m}{\sum} L ~\| \nabla f_i(x^*)\| \leq \hM 
	$$
by the definitions of $L$ and $\hM$ from \eqref{def-L} and \eqref{def-M}. This shows that the rate results obtained in Theorem \ref{theo-rate-quadratics} with decaying stepsize for the special case of quadratic functions is tighter than that of the general case of smooth functions obtained in Theorem \ref{theo-ig-rate-general} as expected.
\end{rema}

Under a strong convexity-type condition and subgradient boundedness, N\'edic and Bertsekas consider the IG method with constant stepsize and show that when $f_i$ are convex but not necessarily smooth or differentiable, for any given $\varepsilon>0$, it suffices to have $\bigO({\log(\frac{1}{\varepsilon})}/{\varepsilon^2}) $ cycles of IG for converging to the $\varepsilon$-neighborhood $\{x \in \R^n: \|x-x^*\| \leq \varepsilon\}$ of an optimal solution $x^*$ \cite[Proposition 2.4]{NedicBert2001IncSubgrad}. 
The following Corollary of Theorem \ref{theo-ig-rate-general} shows that this result can be improved to $\bigO({\log(\frac{1}{\varepsilon})}/{\varepsilon})$ when the component functions are smooth and strongly convex.
\begin{corollary}\label{coro-ig-constant-complexity} Under the conditions of Theorem \ref{theo-ig-rate-general}, let $\varepsilon < 2\hM/c^2 $ be a given positive number. IG with constant stepsize $\alpha = {\varepsilon c} / {(2\hM) }$ requires at most 
            \beq \bigO\bigg(\frac{\hM}{c^2} \frac{\log(1/\varepsilon)}{\varepsilon} \bigg) \label{def-K} 
            \eeq
cycles to guarantee convergence to an $\varepsilon$-neighborhood of the optimal solution $x^*$. 
\end{corollary}
\begin{proof} Given such $\varepsilon>0$ and stepsize $\alpha$, we note that $c\alpha<1$ and ${\alpha  \hM}/{c} = \varepsilon/2$. Furthermore, by Theorem \ref{theo-ig-rate-general}, the inequality \eqref{const-step-recursion} holds with $M_\infty$ replaced by $\hM$.  Therefore, there exists a constant $K$ such that 
   $$(1 - c\alpha)^k \dist_1 \leq \exp(- c\alpha k) \dist_1 < \frac{\varepsilon}{2}, \quad \forall k\geq K,$$
and $\dist_{k+1}< \varepsilon$ for all $k\geq K$, i.e. the iterates lie inside an $\varepsilon$-neighborhood of the optimizer after $K$ cycles. By taking $\log$ of both sides and using $\log(1-z) \approx z$ for $z$ around zero, straightforward calculations show that this condition is satisfied for $K$ satisfying \eqref{def-K}. 
\end{proof}

The rate results in Theorem \ref{theo-rate-quadratics} for $\bigO(1/k)$ stepsize (when $s=1$) requires adjusting the stepsize parameter $R$ to the strong convexity parameter $c$ which requires the estimation or the knowledge of a lower bound for $c$ (need to choose $R>1/c$). The following example illustrates that the convergence can be slow when $R$ is not properly adjusted to $c$. Similar issues with $1/k$-decay step sizes are also widely noted in the analysis of the stochastic gradient descent method in the stochastic approximation literature, see e.g. \cite{Chung:1954iy, Nemirovski:2009kb,Moulines:2011vy, Bach:2014uc}. 

\begin{example}\label{exam-slow-conv}  Let $f_i(x) = x^2/20$ for $i=1,2$, $x\in \R$. Then, we have $m=2$, $c = 1/5$ and $x^*=0$. Take $R=1$ which corresponds to the stepsize $1/k$. The IG iterations are
    $$ x_1^{k+1} = \bigg(1 - \frac{1}{10k}\bigg)^2 x_1^k.$$
If $x_1=1$, a simple analysis similar to \cite{Nemirovski:2009kb} shows $x_1^{k} = \hbox{dist}_k > \Omega(\frac{1}{k^{1/5}})$.
\end{example}


\subsection{Lower bounds}
\label{subsec-least-squares}

Consider the following set of quadratic functions which are stronly convex with parameter $c$ and have Lipschitz gradients with constant $L$: 
  \beqas \Cmnsub = \bigcup_{n=1}^\infty \bigg \{\bar{f}(x) = \frac{1}{2} x^T P x - q^Tx + r ~\bigg | ~ P\mbox{ symmetric},  
     \hesslb \Idn \preceq P \preceq L \Idn; x,q \in \R^n ; r \in \R \bigg \}.
   \eeqas
Theorem \ref{theo-rate-quadratics} and Remark \ref{rema-compare-constants-to-quad-case} shows that when IG is applied to quadratic functions $\bar{f}_i:\R^n \to \R$ with a sum $\bar{f} \in \Cmnsub$ using a stepsize $\alpha_k = R/k$ where $R$ is properly chosen as a function of the strong convexity constant $c$, it results in 
    $$\limsup_{k\to\infty} k \dist_k \leq \frac{{M}}{Rc-1} \leq \frac{\hM}{Rc-1} = \frac{LGm}{Rc-1}. $$
In other words, $\dist_k = \bigO(1/k)$. A natural question is whether one could improve this rate by choosing a perhaps alternative stepsize. Would it be possible to obtain $\dist_k = o(1/k)$ uniformly (for every such $\{\bar{f_i}\}_{i=1}^m$ and $m$)? The next result gives a negative answer to this question showing that no matter which stepsize we choose, there exists simple quadratic functions which result in iterates $\{x_1^k\}$  satisfying $\dist_k \geq \tilde{\Omega}(1/k)$. 
\begin{theorem}\label{theo-lower-bound-ig}  Consider the following IG iterations applied to quadratic component functions $\bar{f}_i:\R^n \to \R$ where $\bar{f}=\big(\sum_{i=1}^m \bar{f}_i\big) \in \Cmnsub$:
   $$x_{i+1}^k = x_i^k - \sigma(c, L, k) \nabla \bar{f}_i(x_i^k), \quad k\geq 1, \quad i=1,2,\dots,m,$$
where the stepsize sequence $\alpha_k = \sigma(c, L, k):\R_+^3 \to \R_+$ is independent from the choice of each $\bar{f}_i$. Suppose that for every choice of $m,n$ and such $\{ \bar f_i : \R^n \to \R \}_{i=1}^m$, we have 
	$$\limsup_{k \to \infty} k\dist_k \leq \bar b $$	
where $\bar b>0 $ depends only on $L, G, m, c$ and $\sigma$. Then, the following statements are true:
\begin{enumerate}
	\item The stepsize sequence satisfies $\limsup_{k\to\infty} k \alpha_k \geq \underline b$ where $\underline b = \bar{b} / L $.
	\item There exists positive integers $\tilde{m}$, $\tilde{n}$ and functions $\{\tilde{f}_i : \R^{\tilde{n}} \to \R \}_{i=1}^{\tilde{m}}$ such that $\tilde{f}=\big(\sum_{i=1}^{\tilde{m}} \tilde{f}_i\big) \in \Cmnsub$ and the iterates $\{x_1^k\}$ generated by the IG applied to $\tilde{f}=\sum_{i=1}^m \tilde{f}_i$ satisfy 
		$$ \dist_k \geq \tilde{\Omega} ( 1/k). $$
\end{enumerate}

\end{theorem}

\begin{proof} \
	\begin{itemize}
			\item [$(i)$] We follow the approach of \cite[Appendix A]{Rakhlin:2011} which was introduced to justify the optimality of the $\Theta(1/k)$ stepsize for the stochastic gradient descent method. Consider the simple example $\bar f(x) =\bar f_1(x) = \frac{L}{2}x^2 \in \Cmnsubinf$ with only one component function in dimension one ($m=n=1$) or a similar alternative example $\bar f(x) =\frac{L}{2}\big(x(1)^2+x(2)^2\big) \in \Cmnsubinf$ with two component functions $\bar f_1(x)=\frac{L}{2}x(1)^2$ and $\bar f_2(x)=\frac{L}{2}x(2)^2$ in dimension two ($n=m=2$) where $x(\ell)$ denotes the $\ell$-th coordinate of the vector $x$. In any of these two examples, IG becomes the classical gradient descent method leading to the iterations $x_1^{k+1}=\prod_{j=1}^k(1 - \alpha_j L) x_1^1$. By the assumption, we have at least
   \beq\label{neces-cond-for-1overk} \bigg | \prod_{j=1}^k (1 - \alpha_j L) \bigg | \leq \frac{\bar b}{k} + o(\frac{1}{k}) \leq \frac{2 \bar b}{k} \quad \mbox{for } k \mbox{ large}
   \eeq
and $\alpha_k \to 0$ (otherwise simple examples show the global convergence may not happen from an arbitrary initial point). 
By taking the natural logarithm of both sides, this is equivalent to requiring
  $$ \sum_{j=1}^k -\ln | 1-\alpha^j L | \geq 2 \bar b \log k \quad \mbox{for } k \mbox{ large}. $$
Using $ 2z \geq - \ln(1-z) $ for $0\leq z \leq \frac{1}{2}$, it follows that  
	\beq \sum_{j=1}^k \alpha_j \geq \underline{b} \log k\label{ineq-to-contradict}
	\eeq  
when $k$ is large enough. Assume there exists $\delta$ such that $ \limsup_{k\to\infty} k\alpha_k < \delta < \underline{b}  $. Then, by definition of the limit superior, we have $\alpha_k \leq \frac{\delta }{k}$ for any $k$ large enough. By summing this inequality over the iterations $k$, we obtain $\sum_{j=1}^k \alpha_j \leq \delta \log(k) + b_2$ for a constant $b_2$ and for any $k$ large enough. This contradicts with \eqref{ineq-to-contradict}. Therefore, no such $\delta$ exists, i.e. $\limsup_{k\to\infty} k \alpha_k \geq \underline{b}$. This completes the proof. 
		\item [$(ii)$] Consider the following simple example with two quadratics $\bar{f} = \bar{f_1} + \bar{f}_2$ with $\bar{f}_1(x) = \frac{L}{2}(x-1)^2$ and $\bar{f}_2(x) = \frac{L}{2}(x+1)^2$ in dimension one ($m=2$, $n=1$). Then, applying IG with an initial point ${x}_1^1 \in \R$ results in the iterates $\{{x}_1^k, {x}_2^k\}$ with
  \beqa 
            {x}_2^k  &=& {x}_1^k - \bar{\alpha}_k ({x}_1^k - 1), \label{x2k-evolution-example}\\
            {x}_1^{k+1} &=& (1-\bar{\alpha}_k)^2 {x}_1^k - (\bar{\alpha}_k)^2,  \label{x1k-stays-positive}\\
            {x}_2^{k+1} &=& (1-\bar{\alpha}_k)^2 {x}_2^k + (\bar{\alpha}_k)^2 \label{x2k-stays-positive},
   \eeqa
where $\bar{\alpha}_k = \alpha_k L$ is the normalized stepsize. Define $y^k = {x}_1^k + {x}_2^k$. By summing up \eqref{x1k-stays-positive} and \eqref{x2k-stays-positive}, we see that
     \beq\label{yk-evolution} y^{k+1} = (1-\bar{\alpha}_k)^2 y^k = \prod_{j=1}^k (1-\bar{\alpha}_j)^2 y^1.
     \eeq
By the necessary condition \eqref{neces-cond-for-1overk}, we also have
       \beq 0 \leq |y^k| \leq \bigO(1/k^2). \label{yk-conv-speed}
       \eeq
Finally, plugging  $y^k = {x}_1^k + {x}_2^k$ into \eqref{x2k-evolution-example}, we obtain
   $${x}_1^k = \frac{y^k}{2} + \bar{\alpha}_k \frac{({x}_1^k - 1)}{2}. $$   
As $\alpha_k = \tilde{\Omega}(1/k)$ by part $(i)$ and ${x}_1^k$ is converging to zero, it follows from \eqref{yk-conv-speed} and the triangle inequality that
   $$ |{x}_1^k| = \dist_k \geq  \bar{\alpha}_k \frac{| {x}_1^k - 1|}{2} - \frac{|y^k|}{2} =\tilde{\Omega}(1/k).     $$ 
This completes the proof.   
	\end{itemize}

\end{proof}

\subsubsection{Lower bounds for stepsize $\alpha_k = R/k^s$ with $s\in[0,1]$} In this section, we are interested in the number of cycles necessary with IG to reach to an $\varepsilon$-neighborhood of an optimal solution using the stepsize $\alpha_k = R/k^s$ with $s\in[0,1]$. As before, we consider the case when the  component functions are smooth and their sum is strongly convex. 

When $s=0$, the stepsize $\alpha_k = \alpha$ is a constant and there exists simple examples (with two quadratics in dimension one) which necessitate $\Omega\big(\log(1/\varepsilon)/\varepsilon\big)$ cycles to reach to an $\varepsilon$-neighborhood of an optimal solution (see \cite[Proposition 2.2]{Luo1991convergence}). Therefore, it can be argued that the dependancy on $\varepsilon$ of the iteration complexity in Corollary \ref{coro-ig-constant-complexity} cannot be improved further. 

For the special case $s \in (1/2,1]$, Luo gives an analysis which suggests that one would expect to have $\dist_k = \Omega (1/k^s)$ for least square problems (see 
\cite[Remark 2, after the proof of Theorem 3.1]{Luo1991convergence}). We next provide a rigorous lower bound for the more general case $s\in(0,1)$. Specifically, the simple example given in the proof of part $(ii)$ of Theorem \ref{theo-lower-bound-ig} provides a lower bound for $s \in (0,1]$ by an analysis almost identical to the proof of part $(ii)$ of Theorem \ref{theo-lower-bound-ig} leading to the following result: 

\begin{lemma}\label{lemma-lower-bd-s} Consider the iterates $\{x_1^k\}$ generated by the IG method with decaying stepsize $\alpha =R/k^s$ where $s\in (0,1]$ applied to quadratic component functions $f_i$ whose sum $f$ is strongly convex. Then, $\dist_k =\Omega(1/k^s)$.
\end{lemma}
\todo{\hfill State like \\ \hfill $\varepsilon^{1/s}$ instead?}

This lower bound is based on an example in dimension one. However, one can also construct similar examples in higher dimensions. In Appendix \ref{sec-appendix-lower-bound}, we provide an alternative example in dimension two for illustrating this fact, although it is not as simple as the dimension one example.

\section{Convergence rate analysis for IN}

To analyze the gradient errors introduced in the IN iterations \eqref{eq-inner-update-2}--\eqref{eq-hessian-update}, we rewrite the outer IN iterations using \cite[Equation 2.12]{GurOzPa15} as: 
\beq\label{eq-iter-newton-via-error} x_1^{k+1} = x_1^k - \alpha_k (\bar{H}_m^k)^{-1}(\nabla f(x_1^k) + e_g^k ) 
\eeq
where
\beq\label{eq-grad-error-alg3}
e_g^k = \sum_{j=1}^m \bigg(\nabla f_j (x_j^k)-\nabla f_j (x_1^k) + \frac{1}{\alpha_k k} \nabla^2 f_j(x_j^k)(x_1^k - x_j^k) \bigg)
\eeq
is the gradient error and  
\beq  \bar{H}_m^k =  \frac{H_0^1 + \sum_{i=1}^k  \sum_{j=1}^m \nabla^2 f_j(x_j^i)  }{k} = \frac{ \sum_{i=1}^k  \nabla^2 f(x_1^i)  }{k} + e_h^k\label{eq-incr-newton-iter-identities}
\eeq
is an averaged Hessian up to an error term
\beq  e_h^k = \frac{H_0^1 + \sum_{i=1}^k  \sum_{j=1}^m \bigg( \nabla^2 f_j(x_j^i) - \nabla^2 f_j(x_1^i)\bigg) }{k}. \label{eq-incr-newton-hessian-error}
\eeq

\noindent We let $\alpha_k = R/k$ and introduce the following norm: 
	\beq \|z\|_{*} := \big(z^T H_* z \big)^{1/2}, \quad z \in \R^n \quad \mbox{where} \quad H_* := \nabla^2 f(x^*) .
		\label{def-star-norm-H-star}
	\eeq
which arises in the analysis of the self-concordant functions and Newton's method \cite{nesterov2004introductory}.  The next theorem shows that unlike IG, IN can achieve the $\bigO(1/k)$ rate without requiring to know or estimate the strong convexity constant of $f$. Furthermore, the constants arising in IN when considered in the $*$-norm do not depend on the Lipschitz constant $L$ unlike IG. 
\begin{theo}\label{theo-incr-newt-no-cond-number} Let $f_i$ be component functions satisfying Assumptions \ref{assump-sum-is-str-cvx} and \ref{assum-C2}. Consider the the iterates $\{x_1^k, \dots, x_m^k\}$ generated by the IN method with stepsize $\alpha_k = R/k$ where $R>1$. Assume that the iterates are uniformly bounded. Then, we have
\beq
   \limsup_{k\to\infty} k \|x_1^k - x_* \|_* \leq \frac{ B R(R+1)}{R-1}
\eeq
where $ \| \cdot \|_*$ and $H_* $  are defined by \eqref{def-star-norm-H-star} and $B= \sum_{i=1}^m \|H_*^{-1/2} \nabla f_i(x^*) \| \leq G/\sqrt{c}$ where $G$ is defined by \eqref{def-G}.
\end{theo}

The proof of this Theorem is given in the Appendix \ref{sec-appendix-proof-newton}. The main idea is to change variables $y=H_*^{1/2}x$ and analyze the corresponding iterates $y_1^k = H_*^{1/2}x_1^k$. By this change of variables, it can be shown that $\{y_1^k\}$ follows a similar recursion to the IN iterates $\{x_1^k\}$ converging to $y_* = H_*^{1/2}x_*$. Then, one can analyze how fast the sequence $\|y_1^k - y_*\| = \| x_1^k - x_*\|_*$ decays to zero by exploiting the fact that $y$-coordinates have the advantage that the local strong convexity constant and the local Lipschitz constant of $f$ around $y_*$ are both equal to one due to the normalization obtained by this change of variable.

\section{Conclusion}
We analyzed the convergence rate of the IG and IN algorithms when the component functions are smooth and the sum of the component functions is strongly convex. This covers the interesting case of many regression problems including the $\ell_2$ regularized regression problems. For IG, we show that the distance of the iterates converges at rate $\bigO(1/k^s)$ to the optimal solution with a diminishing stepsize of the form $\alpha_k = \bigO(1/k^s)$ for $s \in (0,1]$. This improves the previously known $\bigO(1/\sqrt{k})$ rate (when $s \in (1/2,1]$) and translates into convergence at rate $\bigO(1/k^{2s})$ of the suboptimality of the objective value. For constant stepsize, we also improve the existing iteration complexity results for IG from $\bigO(\frac{\log(1/\varepsilon)}{\varepsilon^2})$ to $\bigO(\frac{\log(1/\varepsilon)}{\varepsilon})$ to reach to an $\varepsilon$-neighborhood of an optimal solution. In addition, we provided lower bounds showing that these rates cannot be improved using IG.

Achieving the fastest $\bigO(1/k)$ rate in distances with IG requires a good knowledge or approximation of the strong convexity constant of the sum function $f$. However, we showed that IN as a second-order method, can achieve this fast rate without the knowledge of the strong convexity constant.  

\appendix 
\section{Proof of Theorem \ref{theo-incr-newt-no-cond-number}}\label{sec-appendix-proof-newton} 
\begin{proof}
By a change of variable let $y=H_*^{1/2}x$ and define $\hf(y) = f(x) $. Consider the IN iterates in the $y$-coordinates.  By the chain rule, we have
   \beqa \nabla f(x) = H_*^{1/2}\nabla \hf(y), \quad
   \nabla^2 f(x) = H_*^{1/2}\nabla^2 \hf(y) H_*^{1/2}. \label{change-of-variable}
   \eeqa
 Using these identities, the IN iterations \eqref{eq-inner-update-2}--\eqref{eq-hessian-update} become
\begin{equation}\label{eq-inner-update-2bis} y_{i+1}^k := y_i^k - \alpha_k (\bar{D}_{i}^k)^{-1} \nabla \hf_i (y_i^k), \quad i=1,2,\dots,m, 
       \end{equation}
       where $\bar{D}_i^k = D_i^k/k$ with 
       \begin{equation}\label{eq-hessian-update-bis} y_i^k = H_*^{1/2}x_i^k, \quad D_{i}^k := D_{i-1}^k + \nabla^2 \hf_i(y_i^k) = H_*^{-1/2} H_i^k H_*^{-1/2}.
       \end{equation}    
Furthermore, the IN method is globally convergent under these assumptions (see \cite{GurOzPa15}), i.e. $x_1^k \to x^*$. More generally, due to the cyclic structure, we have also $x_i^k \to x^*$ for each $i=1,2,\dots,m$. Then, from the Hessian update formula \eqref{eq-hessian-update}, it follows that $\bar{H}_i^k \to H_*$ for each $i=1,2,\dots,m$ fixed and
    \beq \bar{D}_i^k \to \nabla^2 \hf(y^*) = H_*^{-1/2} H_* H_*^{-1/2} = \Idn, \quad i=1,2,\dots,m, \label{averaged-Dik-converges}
    \eeq
where we used the second change of variable identity from \eqref{change-of-variable} to calculate $\nabla^2 \hf(y^*)$.       Comparing the IN iterations \eqref{eq-inner-update-2} in the $x$-coordinates and the IN iterations \eqref{eq-inner-update-2bis} in the $y$-coordinates, we see that they have exactly the same form, the only differences are that in the latter the gradients and the Hessian matrices are taken with respect to $y$ (instead of $x$) and $f$ is replaced with $\hf$. Therefore, the inequalities \eqref{eq-iter-newton-via-error}--\eqref{eq-grad-error-alg3} hold if we replace $f$ with $\hf$ and $x_j^i$ with $y_j^i$ leading to \beqa\label{eq-iter-newton-via-error-bis} y_1^{k+1} = y_1^k - \alpha_k (\bar{D}_k)^{-1}(\nabla f(y_1^k) + e_y^k ), \quad \mbox{where} \quad \bar{D}_k := \bar{D}_m^k,  
\eeqa
and the gradient error becomes
\beq\label{eq-grad-error-alg3-bis}
e_y^k = \sum_{j=1}^m \bigg(\nabla \hf_j (y_j^k)-\nabla \hf_j (y_1^k) + \frac{1}{R} \nabla^2 \hf_j(y_j^k)(y_1^k - y_j^k) \bigg)
\eeq
where we set $\alpha_k =R/k$. Setting $\nabla \hf(y_1^k) =  Y_k (y_1^k - y^*)$ in \eqref{eq-iter-newton-via-error-bis} with an averaged Hessian 
	\beq Y_k = \int_0^1 \nabla^2 \hf(y^*+\tau (y_1^k-y^*)) d\tau
	    \label{def-Ayk}
	\eeq
where $y^* = H_*^{1/2}x_*$ and using the triangle inequality we obtain
\beq
\|y_1^{k+1} - y^*\| \leq \underbrace{\big\| \big(\Idn- \frac{R}{k} \bar{D}_k^{-1} Y_k \big) (y_1^k - y^*) \big\|}_{\text{$:=m_k$}} + \frac{R}{k} \underbrace{ \| (\bar{D}_k)^{-1} e_y^k\|}_{\text{$:=n_k$}}. \label{dist-iter-in-y}
\eeq
The remaining of the proof consists of estimating the terms $m_k$ and $n_k$ on the right-hand side separately in the following three steps which gives the  desired convergence rate of the left-hand side $\|y_1^{k+1} - y^*\| = \| x_1^{k+1} - x^*\|_*$:  
\begin{itemize}
   \item [{$\mbox{Step 1}$:}] (Bounding $m_k$) We first observe that
\beq 
m_k^2 = \bigg\| \big(\Idn- \frac{R}{k} \bar{D}_k^{-1} Y_k \big) (y_1^k - y^*) \bigg\|^2 = (y_1^k - y^*)^T S_k (y_1^k - y^*)
\eeq
where 
 \beq S_k = \Idn - \frac{R}{k} Z_k, \quad  Z_k = Y_k \bar{D}_k^{-1} + \bar{D}_k^{-1}Y_k - \frac{R}{k} Y_k \bar{D}_k^{-2}Y_k. \label{def-Mk-Sk}
 \eeq
From \eqref{averaged-Dik-converges}, we have  $\bar{D}_k = \bar{D}_m^k \to  \Idn$. Furthermore, as $y_1^k$ converges to $y^*$ by the global convergence of IN, $Y_k$ defined in \eqref{def-Ayk} converges to $\nabla_y^2 f(y^*) = \Idn$ as well. Therefore, we have $Z_k = 2\Idn + o(1)$ in \eqref{def-Mk-Sk} which leads to 
   $$S_k = (1- \frac{2R}{k})\Idn + o(\frac{1}{k}). $$
Then, for every $\varepsilon \in (0,1) $, there exists a finite $k_1=k_1(\varepsilon)$ such that for $k \geq k_1(\varepsilon)$, 
   $$S_k \preceq \bigg(1-\frac{2R (1- \varepsilon)}{k}\bigg)\Idn  $$
and therefore 
\beqas
m_k^2 = (y_1^k - y^*)^T S_k (y_1^k - y^*) 
   \leq \bigg(1-\frac{2R(1-\varepsilon)}{k}\bigg) \| y_1^k - y^*\|^2
\eeqas
for $k\geq k_1(\varepsilon)$. By taking the square roots of both sides, for $k\geq \max\{k_1,2R\}$, we obtain
    \beq m_k \leq \bigg(1-\frac{R(1-\varepsilon)}{k}\bigg) \| y_1^k - y^*\| \label{bound-mk-final}
    \eeq  
where we used $(1-z)^{1/2} \leq 1 - z/2$ for $z \in [0,1]$ with $z=\frac{2R(1-\varepsilon)}{k}$.
\item [{$\mbox{Step 2}$:}]
\noindent (Bounding $n_k$) Similarly we can write 
   $$\nabla \hf_j (y_j^k)-\nabla \hf_j (y_1^k) = Y_{k,j} (y_j^k - y_1^k)$$ with an averaged Hessian 
    \beq Y_{k,j} = \int_0^1 \nabla^2 \hf_j(y_1^k +\tau (y_j^k-y_1^k)) d\tau \underset{k\to\infty}{\to} \nabla^2 \hf_j(y^*) \preceq \sum_{i=1}^m \nabla^2 \hf_i(y^*)  = \nabla^2 \hf(y^*)=\Idn \label{A-yjk-limit}
    \eeq
as $k \to \infty$ for all $j=1,2,\dots,m$ where we used \eqref{averaged-Dik-converges} in the last equality and the fact that $\nabla^2 f_i (y^*) \succeq 0$ implied by the convexity of $f_i$. Next, we decompose the gradient error term \eqref{eq-grad-error-alg3-bis} into two parts as
$$
  e_y^k = e_{y,1}^k + e_{y,2}^k
$$ 
with
$$ e_{y,1}^k = \sum_{j=1}^m  Y_{k,j} (y_j^k - y_1^k),  \quad e_{y,2}^k =\frac{1}{R} \sum_{j=1}^m  \nabla^2 \hf_j(y_j^k) (y_1^k - y_j^k).
$$
From the triangle inequality for $n_k$ defined in \eqref{dist-iter-in-y}, we have
  \beq n_k  \leq \sum_{\ell=1}^2 n_{k,\ell}  \quad \mbox{with} \quad n_{k,\ell} = \|\bar{D}_k^{-1} e_{y,\ell}^k\|.\label{nk-triangle-bd}
  \eeq
We then estimate $n_{k,\ell}$ for $\ell=1$ and $\ell=2$:
  \beqa
  n_{k,1} &=& \| \sum_{j=1}^m \bar{D}_k^{-1} Y_{k,j} (y_j^k - y_1^k) \| \nonumber \\
               &=& \frac{R}{k} \bigg\| \sum_{j=1}^m \sum_{\ell=1}^{j-1} \bar{D}_k^{-1} Y_{k,j}  \big(\bar{D}_\ell^k \big)^{-1} \nabla \hf_\ell (y_\ell^k) \bigg\|.  \label{control-grad-error-y-one}
  \eeqa
From \eqref{averaged-Dik-converges} and \eqref{A-yjk-limit}, for every $\ell, j \in \{1,2, \dots,m\}$, each summand above in the last equality satisfies
\beqa \lim_{k \to \infty} \bar{D}_k^{-1} Y_{k,j} \big(\bar{D}_\ell^k \big)^{-1} \nabla \hf_\ell (y_\ell^k)    
   =  \nabla^2 \hf_j(y^*)  \nabla \hf_\ell (y^*) 
  \eeqa
so that  
\beqas 
   \lim_{k\to\infty}  k n_{k,1} &=& R \bigg\| \sum_{\ell=1}^m \sum_{j=\ell+1}^{m}  \nabla^2 \hf_j(y^*)  \nabla \hf_\ell (y^*) \bigg\| \\
         &\leq&  R \sum_{\ell=1}^m \big\|\sum_{j=\ell+1}^{m} \nabla^2 \hf_j(y^*) \big\| \| \nabla \hf_\ell (y^*)\| \leq  R \sum_{\ell=1}^m \| \nabla^2 \hf(y^*) \| \| \nabla \hf_\ell (y^*)\| \\
         &\leq& R B
\eeqas   
where in the last step we used the fact that $\nabla^2 \hf (y^*) = \Idn$ and the change of variable formula \eqref{change-of-variable} on gradients.
Similarly, 
  \beqa n_{k,2} = \|\bar{D}_k^{-1} e_{y,2}^k \| &=& \frac{1}{R} \bigg \| \sum_{j=1}^m \bar{D}_k^{-1}   \nabla^2 \hf_j(y_j^k) (y_j^k - y_1^k) \bigg\| \nonumber \\
  &=& \frac{1}{k} \bigg\| \sum_{j=1}^m \bar{D}_k^{-1}   \nabla^2 \hf_j(y_j^k)  \sum_{\ell=1}^{j-1} (\bar{D}_\ell^k)^{-1}\nabla \hf_\ell(y_\ell^k) \bigg\| \nonumber
  \eeqa  
Then, as $\nabla^2 \hf_j(y_j^k) \to \nabla^2 \hf_j(y^*)$, it follows similarly from \eqref{averaged-Dik-converges} that
   \beq \lim_{k\to\infty} k n_{k,2} = \big \| \sum_{j=1}^m\sum_{\ell=1}^{j-1} \nabla^2 \hf(y^*) \nabla \hf(y^*) \big\| \leq B. \label{limsup-bound-2} 
   \eeq  
Going back to the triangle inequality bound \eqref{nk-triangle-bd} on $n_k$, we arrive at
   \beqas \lim_{k \to \infty} k n_k &\leq& \limsup_{k \to \infty} k n_{k,1} + \limsup_{k \to \infty} k n_{k,2} \\ 
       &\leq& RB +  B =  (R+1)B.
   \eeqas
In other words, for any $ \varepsilon>0$, there exists $k_2=k_2(\varepsilon)$ such that   
   \beq n_k \leq (1+\varepsilon)(R+1)B \frac{1}{k}, \quad \forall k\geq k_2(\varepsilon). \label{bound-nk-final}
   \eeq
   
  \item [{$\mbox{Step 3}$:}] (Deriving the rate) Let $\varepsilon \in (0, \frac{R-1}{2R})$ so that $R_\varepsilon:=R(1-\varepsilon) > 1$. Then, it follows from \eqref{dist-iter-in-y}, \eqref{bound-mk-final} and \eqref{bound-nk-final} that for $k\geq \max\{k_1(\varepsilon), 2R, k_2(\varepsilon)\}$, 
 \beqas 
   \|y_1^{k+1} - y^*\|  
        \leq \big(1 - \frac{R_\varepsilon}{k}\big) \| y_1^k - y^*\| +  \frac{(1+\varepsilon) B R(R+1)}{k^2}. 
  \eeqas
Applying Lemma \ref{lem-chung-1} with $u_k = \|y_1^k - y^*\| = \| x_1^{k} - x^*\|_*$, $a=R_\varepsilon > 1$ and $s=1$ leads to 
\beq
   \limsup_{k\to\infty} k \|x_1^k - x_* \|_* \leq \frac{(1+\varepsilon) B R(R+1)}{R(1-\varepsilon)-1}.
\eeq
Letting $\varepsilon\to 0$, completes the proof.
\end{itemize}
\end{proof}

\section{An example in dimension two with $\dist_k = \Omega(1/k^s)$ }\label{sec-appendix-lower-bound}
The aim is to construct a set of component functions in dimension two such that if IG is applied with stepsize $\alpha_k = \Theta(1/k^s)$ with $0<s\leq 1$, the resulting iterates satisfy $\dist_k \geq \Omega(1/k^s)$. 

Consider the following least squares example in dimension two ($n=2$) with $m=8$ quadratics defined as
\beqas
     \tilde{f}_i(x) &=& \frac{1}{2}(c_i^Tx + 1)^2, \quad i=1,2,\dots,8,
\eeqas
where the vectors $c_i \in \R^2$ are 
\beqa    
 c_1 = c_6  &=& - c_2 = -c_5 = [-1, 0]^T ,\label{exam-ci-symmetry-1} \\
 c_3 = c_8 &=& - c_4 = -c_7 = [0, -1]^T \label{exam-ci-symmetry-2} .
\eeqa     
It is easy to check that the sum $\tilde{f}:=\sum_{i=1}^8 \tilde{f}_i$ is strongly convex as 
   \beq \nabla^2 \tilde{f}(x) = \sum_{i=1}^8 c_i c_i^T = 4 \Idn \succ 0.\label{hess-of-quad}
   \eeq 
Starting from an initial point $\tilde{x}_1^1$, the IG method with stepsize $\alpha_k$ leads to the iterations
    \beq \tilde{x}_{i+1}^k = (\Idn - \alpha_k c_i c_i^T)\tilde{x}_i^k - \alpha_k c_i, \quad i=1,2,\dots,8, \label{exam-ig-iters}
    \eeq
which implies 
    \beqa \tilde{x}_1^{k+1} &=&\prod_{i=1}^8 (\Idn - \alpha_k c_i c_i^T) \tilde{x}_1^k - \alpha_k \sum_{i=1}^8 c_i + \alpha_k^2 \sum_{1\leq i < j \leq 8} (c_j^T c_i) c_j + \bigO(\alpha_k^3) \nonumber \\
     \tilde{x}_1^{k+1} &=& \prod_{i=1}^8 (\Idn - \alpha_k c_i c_i^T) \tilde{x}_1^k + \bigO(\alpha_k^3)  \label{eight-quadratics-distances}
    \eeqa
where in the second step we used the fact that the terms with $\alpha_k$ and $\alpha_k^2$ above vanish due to symmetry properties imposed by relations \eqref{exam-ci-symmetry-1} and \eqref{exam-ci-symmetry-2}. The cyclic order $\{1,2,\dots,8\}$ is special in the sense that it takes advantage of the symmetry in the problem leading to cancellations of the $\bigO(\alpha_k)$ and $\bigO\big(\alpha_k^2\big)$ terms leading to smaller $\bigO\big(\alpha_k^3\big)$ additive error terms, whereas it can be checked that this is not the case for the order $\{2,3,\dots,8,1\}$. With this intiution in mind, we next show that the sequence $\{\tilde{x}_2^k\}$ converges to the optimal solution $x^* = 0$ slower than the sequence $\{\tilde{x}_1^k\}$ does. 

Using \eqref{hess-of-quad}, the fact that $x^* = 0$ for this specific example and the triangle inequality on \eqref{eight-quadratics-distances}, 
   \beqas \dist_{k+1} &\leq& \|  \prod_{i=1}^8 \big( \Idn - \alpha_k c_i c_i^T \big) \| \dist_k + \bigO(\alpha_k^3)\\ 
    &\leq& \bigl |1 - 4\alpha_k + \bigO( \alpha_k^2) \bigr | \dist_k + h_3 (\alpha_k)^3
    \eeqas
for some constant $h_3>0$. As $\alpha_k = \Theta (1/k^s)$, applying part $(ii)$ of Lemma \eqref{lem-chung-1} with $t=2s$ gives
     \beq \| \tilde{x}_1^k\| =  \bigO(1/k^{2s}). \label{dist-2s-rate}
     \eeq
Then, for $i=1$ the inner iterations \eqref{exam-ig-iters} gives
    $$ \tilde{x}_2^k = \tilde{x}_1^k - \alpha_k (c_1^T \tilde{x}_1^k + 1)  c_1.$$
As $\tilde{x}_1^k \to 0$, $(c_1^T \tilde{x}_1^k + 1)  c_1 \to c_1$. Then, it follows from \eqref{dist-2s-rate} that $\dist(\tilde{x}_2^k) = \|\tilde{x}_2^k\| = \Theta(\alpha_k) = \Theta(1/k^s)$. As the order is cyclic, if we apply IG to the functions with an alternative order $f_1 = \tilde{f_2}$, $f_2 = \tilde{f_3}$,$\dots$, $f_{m-1} =\tilde{f}_m$ and $f_m = \tilde{f}_1$ instead the resulting iterates are $x_1^k = \tilde{x}_2^k$ satisfying $\dist(x_1^k) = \dist(\tilde{x}_2^k) = \Theta(1/k^s)$. We conclude.
      
\bibliographystyle{plain}
\bibliography{distributed_refs}

\end{document}